\documentclass[a4paper,11pt,twoside]{amsart}

\usepackage[latin1]{inputenc}
\usepackage{amsmath, amsfonts, amssymb,amsthm}
\usepackage[mathscr]{eucal} 
\usepackage[pdftex]{graphicx}
\usepackage{color}

\usepackage[T1]{fontenc}
\usepackage[sc]{mathpazo}
\usepackage[all]{xy}
\usepackage{hyperref}

\usepackage{enumerate}

\definecolor{alert}{rgb}{0.8,0,0}

\newcommand{\N}{\mathbb{N}}

\newcommand{\R}{\mathbb{R}}

\newcommand{\h}{\mathbb{H}}

\newcommand{\B}{\mathbb{B}}
\newcommand{\M}{\mathbb{M}}

\newcommand{\Nil}{\mathrm{Nil}_3(\tau)}

\renewcommand{\div}{\mathrm{div}}

\newtheorem{theorem}{Theorem}[section]
\newtheorem{proposition}[theorem]{Proposition}

\newtheorem{lemma}[theorem]{Lemma}
\newtheorem{claim}[theorem]{Claim}

\theoremstyle{definition}
  \newtheorem{definition}[theorem]{Definition}

\theoremstyle{remark}
\newtheorem{remark}[theorem]{Remark}
\newtheorem{example}[theorem]{Example}

\numberwithin{equation}{section}

\title[Jenkins-Serrin Problem]{The Jenkins-Serrin problem for constant mean curvature graphs in the Heisenberg space $\mathbf{ Nil_3}(\tau)$}

\author{Carlos Pe\~{n}afiel}
\address{}
\email{}

\thanks{}



\begin{document}
\maketitle

\begin{abstract}
 In this paper we find functions over bounded domains in the 2-dimensional Euclidean space, whose graphs (in the Heisenberg space) has constant mean curvature different from zero and taking on (possibly) infinite boundary values over the boundary of the domain.   \end{abstract}

\section{Introduction}\label{sec:introduction}

The classical Dirichlet problem for the constant mean curvature and minimal surfaces (H-surfaces and minimal surfaces from now on) in the 3-dimensional Euclidean space $\R^3$, consists in the determination of a function $u=u(x,y)$ satisfying the partial differential equation
\begin{equation}\label{i1}
(1+q^2)r-2pqs+(1+p^2)t=2H(1+p^2+q^2)^{3/2}
\end{equation}
in a fixed domain $\Omega$ in the $(x,y)-$plane, and taking on assigned continuous value on the boundary of $\Omega$.

In order to introduce the Jenkins-Sering theory, we divide the equation (\ref{i1}) in two cases.
\begin{itemize}
\item For the case of minimal surfaces in $\R^3$, the Dirichlet problem was solved by Radô in 1930, for convex domains, see \cite{Rado}. Radô based his proof on the existence theorem for the parametric problem of least area. In \cite{JS} Jenkins and Serrin gave an alternate proof entirely avoiding reference to the parametric problem.

On the other hand, it was well known that in 1934 H. F. Scherk discovered his famous minimal surface which is a graph of a function defined over a square and taking on infinite boundary data. More precisely, Scherk find out the surface given by
$$z=\log\cos x-\log\cos y, \hspace{.2cm}\vert x\vert < \pi /2, \hspace{.2cm} \vert y\vert<\pi/2.$$
The graph of this function is a minimal surface in $\R^3$, the function takes on plus infinite and minus infinite boundary data on alternates sides of the boundary of the square. We call this surface, the Scherk example. This Scherk example can be seem as a solution of the Dirichlet problem for the minimal curvature equation, taking on infinite boundary values over the boundary of the square.   

In \cite{JS} Jenkins and Serrin developed an existence and uniqueness theory applicable to situation in which continuity of the data is set aside and which infinite boundary values are allowed on entire arcs of the boundary. Fundamental for they work was the existence of a solution of the Dirichlet problem  in a convex domain with assigned boundary data, notice that this work generalize the Scherk example. From now on, we call this kind of surfaces, Jenkins-Serrin surfaces and the theory development, the Jenkins-Serrin theory.  

The Jenkins-Serrin theory was extend to $\h\times\R$ (here $\h$ label the 2-dimensional hyperbolic space) by Nelli and Rosenberg in \cite{NelliR}, and to $M^2\times\R$, where $M^2$ is an arbitrary Riemannian surface, by Pinheiro in \cite{Pinheiro1}. And for non-compact domains, the Jenkins-Serrin theory was treated by Mazet L., Rodríguez, M. and Rosenberg, H. in \cite{Mazet}.

Different of the 3-dimensional product spaces $\h\times\R$ and $\mathbb{S}^2\times\R$ (here $\mathbb{S}^2$ label the 2-dimensional Euclidean sphere), which are homogeneous and simply connected, having 4-dimensional isometry groups, we have the Heisenberg space $\Nil$ and the universal cover of the space $PSL_2(\R)$ (here $PSl_2(\R)$ denote the preserving-orientation isometries of the hyperbolic space $\h$) which we denote by $\widetilde{PSL}_2(\R,\tau)$. For more details, see  for instance \cite{T}, \cite{D}.  

The Jenkins-Serrin theory for compact domains in $\Nil$ was treated by Ana Lucia Pinheiro in \cite{Pinheiro2} and for the $\widetilde{PSL}_2(\R,\tau)$ (in the $\tau=-1/2$ case) space by Rami-Younes in \cite{Younes}.

\item On the other hand, for the case $H\neq0$, the Jenkins-Serrin theory for bounded domains was extended to $\R^3$ by Spruck in \cite{Spruck}, for $\h\times\R$ and $\mathbb{S}^2\times\R$ by Hauswirth, Rosenberg and Spruck in \cite{Hauswirth}.

For unbounded domains, the Jenkins-Serrin theory was extended to  $\h\times\R$ by Folha and Mello in \cite{Folha2}, and in the case $M^2\times\R$ (here $M^2$ label a Hadamard surface) by Folha and Rosenberg in \cite{Folha1}.
\end{itemize} 
If we consider complete, simply connected, 3-dimensional homogeneous manifolds having 4-dimensional isometry groups, it result open the case $H\neq0$ for bounded domains in the Heisenberg space $\Nil$ and the $\widetilde{PSl}_2(\R,\tau)$ space. As we have remarked, the fundamental fact in the Jenkins-Serrin theory, was the existence of a solution of the Dirichlet problem with prescribed continuous boundary data, on a convex domain. Since that in \cite{Dajczer} Dajczer and Lira solve this Dirichlet problem for general 3-manifolds carrying on a non-singular Killing field, we use the existence of such solution to extend the Jenkins-Serrin Theory to $\Nil$ in the case $H\neq0$ for bounded domains. The key idea is the study the flux of monotone increasing and decreasing sequences of solutions of the Dirichlet problem  along arcs where they diverge, we will use the Killing submersion in order to study the geometric behaviour of such sequences of graphs. The Jenkins-Serrin problem for unbounded domain in the $\widetilde{PSL}_2(\R,\tau)$ space is treated in \cite{Penafiel-Folha} and for a Killing Submersion, see \cite{Penafiel-Cui}.

The paper is organized as follow, in Section 2 we give the details of the Heisenberg space $\Nil$ seem as a Riamannian submersion over the 2-dimensional Euclidean space $\R^2$. We prove a Maximum principle and we cite the existence theorem of the Dirichlet problem. In Section 3, we deal with $H$-sections, and establishes the Jenkins-Serrin problem as well as we study the properties of the flux of the sequences of solutions. Finally in Section 4, we prove the main theorems. 

\section{Preliminares}
We denote by $\Nil$ the 3-dimensional Lie group endowed with a left invariant metric $g$. For each $\tau\neq0$, $\Nil$ is a homogeneous simply connected Riemannian manifold.

In Euclidean coordinates $\Nil=(\R^3,g)$, where $\R^3$ label the 3-dimensional Euclidean space and $\Nil$ is endowed with the metric 
\begin{equation}\label{e1}
g=dx^2+dy^2+(\tau(ydx-xdy)+dz)^2
\end{equation}
The Lie group $\Nil$ is one of the eight Thurston's geometries (see \cite{T}), and it is well known that there exists a Killing submersion (see \cite{Rosenberg1})
\begin{equation*}
\pi:\Nil\longrightarrow\R^2
\end{equation*}
\begin{equation}\label{e2}
(x,y,z)\longmapsto(x,y)
\end{equation}
from $\Nil$ into the Eucliedan 2-dimensional space $\R^2$. That is $\pi$ is a Riemannian submersion, the bundle curvature is $\tau$ and the unit vector field along the fibers is a Killing vector field. Therefore translations along the fibers are isometries. We denote this Killing vector field by $\xi$.

We call a vector field $Z\in\chi(\Nil)$ vertical if it is a non-zero multiple of $\xi$ and horizontal if $g(Z,\xi)=0.$

In order to obtain an orthonormal frame on $\Nil$, we consider the canonical frame $\{e_1=\partial_x,e_2=\partial_y\}$ of $\R^2$ and consider the horizontal lift $E_1$ and $E_2$ of $e_1$ and $e_2$ respectively. Then the canonical orthonormal frame $\{E_1,E_2,\xi\}$ of $\Nil$ is given by
$$E_1=\partial_x-\tau y\partial_z, \hspace{.5cm} E_2=\partial_y+\tau x\partial_z, \hspace{.5cm} \xi=\partial_z.$$
Denote by $\overline{\nabla}$ the Riemannian connection of $\Nil$, so
\begin{center}
  $\begin{array}{ccc}
\overline{\nabla}_{E_{1}}E_{1}=0 & \overline{\nabla}_{E_{1}}E_{2}=\tau E_{3} & \overline{\nabla}_{E_{1}}E_{3}=-\tau E_{2} \\[10pt]
    \overline{\nabla}_{E_{2}}E_{1}=-\tau E_{3} & \overline{\nabla}_{E_{2}}E_{2}=0 & \overline{\nabla}_{E_{2}}E_{3}=\tau E_{1} \\[10pt]
    \overline{\nabla}_{E_{3}}E_{1}=-\tau E_{2}  & \overline{\nabla}_{E_{3}}E_{2}=\tau E_{1} & \overline{\nabla}_{E_{3}}E_{3}=0
  \end{array}$
\end{center}

$$[E_1,E_2]=2\tau E_3, \hspace{.5cm} [E_2,E_3]=0=[E_1,E_3]$$
for more details see \cite{Dajczer1}, \cite{F}, \cite{D}.

The isometry group of $\Nil$ has dimension 4, the isometries of $\Nil$ are the translations generated by the Killing vector fields
$$F_1=\partial_x+\tau y\partial_z, \hspace{.5cm} F_2=\partial_y-\tau x\partial_z, \hspace{.5cm} F_3=\partial_z$$
and the rotations about the $z$-axis corresponding to 
$$F_4=-y\partial_x+x\partial_y$$
The translations corresponding to $F_1$ and $F_2$ are respectively
$$(x,y,z)\longmapsto(x+t,y,z+\tau ty)$$
and
$$(x,y,x)\longmapsto(x,y+t,z-\tau tx)$$
where $t\in\R$.

\subsection{The mean curvature equation and the maximum principle}
We identify the space $\R^2$ with its lift $\R^2\times\{0\}\subset\Nil$ and for a given domain $\Omega\subset\R^2$, we also denote by $\Omega$ its lift to $\R^2\times\{0\}$.

To a function $u\in C^0(\overline{\Omega})$ on $\Omega$, we define the graph of $u$, denoted by $\Sigma(u)$, as being the set
\begin{equation}\label{e3}
\Sigma(u)=\{(x,y,u(x,y))\in\Nil;(x,y)\in\Omega\}
\end{equation}
Throughout this paper the surface $\Sigma(u)$ will have mean curvature $H>0$ with respect to the upward pointing normal vector of $\Sigma(u)$.

In order to obtain the normal vector $N$ of $\Sigma(u)$, consider the function
$$u^*:\Nil\longrightarrow\R$$
$$(x,y,z)\longmapsto u^*(x,y,z)=u(x,y)$$
and set $F(x,y,z)=z-u^*(x,y,z)$. Therefore, $\Sigma(u)=F^{-1}(0)$, that is $\Sigma(u)$ is the level surface of $F$. It is well known that the function $H$ satisfies
\begin{equation}\label{ee1}
\div_{\Nil}\left(\dfrac{\overline{\nabla}_{g}F}{\vert \overline{\nabla}_gF\vert}\right)=-2H
\end{equation}
where $\div_{\Nil}$ and $\overline{\nabla}_g$ denote the divergence and gradient in $\Nil$. Thus, the upward pointing normal $N$ is given by
$$N=\dfrac{\overline{\nabla}_gF}{\vert\overline{\nabla}_gF\vert}.$$
A straightforward computation shows
$$\overline{\nabla}_gF=-(u_x+\tau y)E_1-(u_y-\tau x)E_2+E_3$$
hence
$$W^2:=\vert\overline{\nabla}_gF\vert=1+(u_x+\tau y)^2+(u_y-\tau x)^2\vert.$$
Using the Riemannian submersion, the function $u$ satisfies the equation ($the$ $mean$ $curvature$ $equation$)
\begin{equation}\label{e4}
\div_{\R^2}\left(\dfrac{\alpha}{W}\partial_x+\dfrac{\beta}{W}\partial_y\right)=2H
\end{equation}
where
\begin{equation}\label{es5}
\alpha=\tau y+u_x, \hspace{.5cm} \beta=-\tau x+u_y, \hspace{.5cm} \textnormal{and} \hspace{.5cm} W^2=1+\alpha^2+\beta^2.
\end{equation}
Thus, the surface $\Sigma(u)$ has mean curvature function $H$ if and only if $u$ is a solution of the following PDE
\begin{equation}\label{e6}
L_H(u):=\dfrac{1}{W}[(1+\beta^2)u_{xx}+(1+\alpha^2)u_{yy}- 2\alpha\beta u_{xy}]-2H=0
\end{equation}
for $\alpha$, $\beta$ and $W$ from (\ref{es5}).

Notice that 
$$N=\dfrac{1}{W}(-\alpha E_1-\beta E_2+E_3)=N^h+N^v$$
where
$$N^h=\textnormal{horizontal part of $N$}=\dfrac{-G(u)}{W}=\dfrac{\alpha E_1+\beta E_2}{W}$$
and
$$N^v=\textnormal{vertical part of $N$}=\dfrac{E_3}{W}.$$
Consider two functions 
$$u,v:\Omega\longrightarrow\R,$$
the upwards pointing normal $N_1$ and $N_2$ of $\Sigma(u_1)$ and $\Sigma(u_2)$ are respectively
\begin{equation}\label{ee2}
\left\{
  \begin{array}{ll}
    N_1=-\dfrac{G(u_1)}{W_1}+\dfrac{1}{W_1}E_3 & \hbox{} \\
    N_2=-\dfrac{G(u_2)}{W_2}+\dfrac{1}{W_2}E_3 & \hbox{} 
  \end{array}
\right.
\end{equation}
where $W_i=W(u_i)$, $i=1,2$.\\
Notice that $\langle W_iN_i, E_3\rangle=1$ and
\begin{eqnarray}\label{ee3}
\nonumber\langle G(u_1)-G(u_2),\dfrac{G(u_1)}{W_1}-\dfrac{G(u_2)}{W_2} \rangle & = & (W_1+W_2)(1-\langle N_1,N_2\rangle) \\
 & = & \dfrac{W_1+W_2}{2}\vert N_1-N_2\vert^2\geq0
\end{eqnarray}
On the other hand
$$G(u_i)=-(u_{ix}+\tau y)E_1-(u_{iy}-\tau x)E_2$$
Setting 
$$\widetilde{X}_i=-\pi_*(G(u_i))=\alpha\partial_x+\beta\partial_y= (u_{ix}+\tau y)\partial_x+(u_{iy}-\tau x)\partial_y$$
and
$$X_{u_i}=\dfrac{\widetilde{X}_i}{W_i}$$
we conclude that
$$X_{u_1}-X_{u_2}=\nabla_0u_1-\nabla_0u_2,$$
where $\nabla_0$ denotes the gradient in the Euclidean space $\R^2$.\\
Using the Riemannian submersion (\ref{ee3}) becomes
\begin{equation}\label{ee4}
\langle\nabla_0u_1-\nabla_0u_2,X_{u_1}-X_{u_2}\rangle_{\R^2}\geq0
\end{equation}
Thus, we have proved the next lemma.
\begin{lemma}\label{l1}
Let $u_1$ and $u_2$ be functions in $C^2(\Omega)$, $\Omega\subset\R^2$ and set $W_i=W(u_i)$, $i=1,2$. Then
\begin{equation}\label{ee5}
\langle\nabla_0u_1-\nabla_0u_2,X_{u_1}-X_{u_2}\rangle\geq0
\end{equation}
with equality at a point if and only if $\nabla_0u_1=\nabla_0u_2$.
\end{lemma}
Following the above notation, the mean curvature function $H$ of the surface $\Sigma(u)$ satisfy the equation
\begin{equation}\label{ee6}
\div_{\R^2}(X_u)=2H
\end{equation}
where $X_u=-\pi_*\left(\dfrac{G(u)}{W}\right).$

For this situation, we prove the following maximum principle.
\begin{theorem}\label{MP}(Maximum principle)
Consider $\Omega\subset\R^2$ a bounded domain. Let $u,v\in C^2(\Omega)$ be two functions whose graphs $\Sigma(u)$ and $\Sigma(v)$ have the same prescribed mean curvature $H$. Let $E\subset\partial\Omega$ be a finite set of points such that $\partial\Omega-E$ consists of smooth arcs, and suppose that $u$ and $v$ extend continuously to each smooth arc of $\partial\Omega-E$. If $u\geq v$ on $\partial\Omega$, then $u\geq v$ on $\Omega$.
\end{theorem}
\begin{proof}
Consider the set $D=\{x\in\Omega,u(x)-v(x)<0\}$. We can translate the surfaces $\Sigma(u)$ and $\Sigma(v)$ in the $\partial_z$ direction to assume that $u<v$ on $\Omega-E$, and by contradiction, we are supposing that $D$ is not empty.

The boundary of $D$ consists of proper curves in $\Omega$ which goes to points of $E$. We can also assume that those curves are regular, ie. $\nabla_0(u-v)$ is nonzero on $\partial D$. Denote by $\widehat{D}$ a connected component of $D$ and let $\widehat{D}_\epsilon\subset\widehat{D}$ be the domain such that $\partial\widehat{D}_\epsilon$ is the set of points in $\partial\widehat{D}$ whose distance from, $E$ is greater from $\epsilon>0$, for $\epsilon$ small enough, together with the union $\cup C_\epsilon$ of circular arcs contained in $\widehat{D}$ which are part of circles centered at points of $\partial\widehat{D}\cap E$, having radius $\epsilon$.  

From the mean curvature equation (\ref{ee6}), we hace
$$\int_{D_\epsilon}\div_{\R^2}(X_u-X_v)=0$$
and by Stokes' Theorem
\begin{equation}\label{ee7}
0=\int_{\widehat{D}_\epsilon}\div_{\R^2}(X_u-X_v)=\int_{\partial\widehat{D}_\epsilon}\langle X_u-X_v,\eta\rangle
\end{equation}
where $\eta$ is the outward unit conormal to $\partial\widehat{D}_\epsilon$.

From Lemma \ref{l1}
\begin{equation}\label{ee8}
\langle X_u-X_v,\nabla_0u-\nabla_0v\rangle\geq0
\end{equation}
As $\nabla_0(u-v)\neq0$ and $u-v\equiv0$ on $\partial\widehat{D}_\epsilon-\cup C_\epsilon$, and $u-v<0$ on $\widehat{D}_\epsilon$, the vector $\nabla_0(u-v)$ is a positive multiple of $\eta$ on $\partial\widehat{D}_\epsilon-\cup C_\epsilon$. Therefore from (\ref{ee8})
$$\int_{\partial\widehat{D}_\epsilon-\cup C_\epsilon} \langle X_u-X_v,\eta\rangle_{\R^2}\geq\delta>0.$$
On the other hand, on $\cup C_\epsilon$, we have
$$\int_{\cup C_{\epsilon}}\langle X_u-X_v,\eta\rangle\leq\int_{\cup C_\epsilon} \vert\langle X_u-X_v,\eta\rangle\vert\leq2\times\textnormal{length}(\cup C_\epsilon)$$
thence, making $\epsilon$tend to zero, we conclude that
$$\int_\epsilon\langle X_u-X_v,\eta\rangle>0,$$
which contradicts (\ref{ee7}). Thus $D$ is empty and $u\leq v$ on $\Omega$.

\end{proof}

\subsection{The Dirichlet problem}
In \cite{Dajczer} Dajczer and Lira have proved the existence and uniqueness of Killing graphs with prescribed mean curvature in Killing submersions. In our context, the surfaces $\Sigma(u)\subset\Nil$ which are the graph of the function
$$u:\Omega\subset\R^2\longrightarrow\R$$
are Killing graphs, since they are transverse to the vertical Killing field $\xi$. From now on, the surfaces $\Sigma(u)$ will be called Killing graphs.
 
For a domain $\Omega\subset\R^2$ with boundary $\Gamma=\partial\Omega$, we denote the correspondent Killing cylinders by $M_0=\pi^{-1}(\overline{\Omega})$ and $K=\pi^{-1}(\Gamma)$. We denote by $H_{cyl}$ the inward mean curvature of $K$ and by $Ric_{\Nil}$ the Ricci tensor of $\Nil$. Then it has showed in \cite[Theorem 1]{Dajczer} the following theorem.
\begin{theorem}\label{DP}(The Dirichelt problem)
Let $\Omega\subset\R^2$ be a domain with compact closure and $C^{2,\alpha}$ boundary. Suppose $H_{cyl}>0$ and 
$$\inf_{\Nil} Ric_{\Nil}\geq-2\inf_\Gamma H^2_{cyl}.$$ 
Let $H\in C^\alpha(\overline{\Omega})$ and $\phi\in C^{2,\alpha}(\Gamma)$ be given functions. If
$$\sup_\Omega\vert H\vert\leq\inf_\Gamma H_{cyl}.$$
Then, there exists a unique function $u\in C^{2,\alpha}(\overline{\Omega})$ satisfying $u\vert_\Gamma=\phi$ whose Killing graph $\Sigma(u)$ has mean curvature $H$. 
\end{theorem}

In this paper the ambient Ricci tensor in the $v$-direction is defined by
$$Ric_{\Nil}(v)=\sum_{i=1}^{2}\langle\overline{R}(w_i,v)v,w_i\rangle$$
where $\overline{R}$ is the curvature tensor and $\{w_1,w_2,v\}$ is an orthonormal basis.
\begin{remark}\label{r1}
Fix a point $p_0\in\Nil$ and take a unit vector $v\in T_{p_0}\Nil$. Let $\Pi$ be the plane orthogonal to the vector $v$. After a rotation around the vertical fiber passing by $p_0$, we can suppose that
$$\Pi=[E_1,aE_2+bE_3], \hspace{.3cm}a^2+b^2=1.$$ 
That is, $\{E_1,aE_2+bE_3\}$ is a orthonormal basis for $\Pi$ and thence 
$$v=-bE_2+aE_3.$$ 
Using the Riemannian connection $\overline{\nabla}$, we have
$$-2\tau^2\leq Ric_{\Nil}(v)\leq2\tau^2.$$
\end{remark}

Denote by $\Gamma\subset\R^2$ a piecewise $C^1$ smooth simple Jordan curve. Taking the parametrization
$$c:[a,b]\longrightarrow\Gamma$$ 
$$s\longmapsto c(s)\in\Gamma$$
consider the geodesic curvature $k(s)$ of $\Gamma$ at the point $c(s)\in\Gamma$, it is well known that when the fiber of the killing submersion are geodesic, then the mean curvature $H_{cyl}$ of the vertical cylinder $\pi^{-1}(\Gamma)$ is given by $$H_{cyl}=\dfrac{k(s)}{2}.$$ 
From the Ricci condition in Theorem \ref{DP} and remark \ref{r1}, $\tau$ must be satisfy
\begin{equation}\label{ee9}
\tau^2\leq\inf_s\left(\dfrac{k(s)}{2}\right)^2
\end{equation}
With this in mind, the Theorem \ref{DP} can be write in the  next form.
\begin{theorem}\label{DP1}(The Dirichelt problem)
Let $\Omega\subset\R^2$ be a domain with compact closure and $C^{2,\alpha}$ boundary. Suppose $H_{cyl}>0$ and $H\geq\vert\tau\vert$ for a constant $H$. Let $\phi\in C^{2,\alpha}(\Gamma)$ be given functions. If
$$2H\leq k(s),$$
where $k(s)$ denotes the geodesic curvature of $\Gamma$.
Then, there exists a unique function $u\in C^{2,\alpha}(\overline{\Omega})$ satisfying $u\vert_\Gamma=\phi$ whose Killing graph $\Sigma(u)$ has mean curvature $H$. 
\end{theorem}

Once there, using Theorem \ref{DP1} we can prove an important theorem for sequences of solutions of the mean curvature equation.
\begin{theorem}\label{CT}(Compactness theorem) 
Let $\{u_n\}$ be a uniformly bounded sequences of solutions of the constant mean curvature equation (\ref{ee6}) in a bounded domain $\Omega$. Then there exists a subsequence which converges uniformly on compact subsets to a solution of (\ref{ee6}) in $\Omega$.\end{theorem}
\begin{proof}
From Theorem \ref{DP1}, we obtain an interior estimate for the first and second derivative as well as to the solutions of the equation (\ref{ee6}). Therefore on compact subdomains we have the equicontinuity of the second derivatives. Consequently by Arzela-Ascoli's Theorem, we obtain a subsequence which converges uniformly on compact subsets to a solution of (\ref{ee6}) in $\Omega$. 
\end{proof}

\section{$H$-sections}
A section is the image of a map  
$$u:\Omega\subset\R^2\rightarrow\Nil,$$ 
such that $\pi\circ u=Id\vert_\Omega$, where $Id$ label the identity map restricted to the domain $\Omega$ and $\pi$ is the canonical projection from $\Nil$ over $\R^2$. Thus, identifying the domain $\Omega$ with its lift to $\R^2\times\{0\}$, we identify the section $u$ with the Killing graph $\Sigma(u)$ and by simplicity we denote this surface by $\Sigma_u$. 

In this section we are going to consider sections $u:\Omega\subset\R^2\rightarrow\Nil$ whose Killing graph $\Sigma_u$ has constant mean curvature (CMC-graphs) $H$. We call such Killing graphs, $H$-sections or $H$-surfaces. For a piece of curve $\gamma\subset\B$, we denote by $k(\gamma)$ its geodesic curvature. It has showed in \cite[Theorem 3.3]{Rosenberg1} the following theorem.
\begin{theorem}\label{S1}
Let $\pi:\Nil\rightarrow\R^2$ the  Killing submersion and let 
$$u:\Omega\rightarrow\Nil$$ 
be an $H$-section over a domain $\Omega\subset\R^2$. Let $U_0$ be a neighbourhood of an arc $\gamma\subset\partial\Omega$ and $\iota:U_0\rightarrow\Nil$ a section. \\ Assume that for any sequence $(x_n)$ of $\Omega$ which converge to a point $x\in\gamma$, the height of $u(x_n)$ goes to $+\infty$, that is, $u(x_n)-\iota(x_n)\rightarrow+\infty$, then $\gamma$ is a smooth curve with $k(\gamma)=2H$. If $H>0$, then $\gamma$ is convex with respect to $\Omega$ if and only if, the mean curvature vector $\overrightarrow{H}$ of $\Sigma_u$ points up along $\Sigma_u$. Moreover, $\Sigma_u$ converges to the vertical $H$-cylinder $\pi^{-1}(\gamma)$ with respect to the $C^k$-topology, for any $k\in\N$. 
\end{theorem}

As a consequence  of this theorem, we prove the following lemma.
\begin{lemma}\label{a15}
Let $u$ be a solution of (\ref{ee6}) in a domain $\Omega$ bounded in part by an arc $\gamma$ and suppose that $m\leq u\leq M$ on $\gamma$. Then, there is a constant $c=c(\Omega)$ such that for any compact $C^2$ sub-arc $\gamma^\prime\subset\gamma$, 
\begin{itemize}
\item[(i)] If $k(\gamma^\prime)\geq 2H$, with strict inequality except for isolated points, there is a neighbourhood $V$ of $\gamma^\prime$ in $\overline{\Omega}$ such that 
$$u\geq m-c$$ 
in $V$.
\item[(ii)] If $k(\gamma^\prime)>-2H$, there is a neighbourhood $V$ of $\gamma^\prime$ in $\overline{\Omega}$ such that 
$$u\leq M+c$$ 
in $V$.
\end{itemize}
\end{lemma}
\begin{proof}
First suppose that $k>2H$ on some sub-arc $\gamma^{\prime}\subset\gamma$, let $p_0$ be the middle point of $\gamma^\prime$ and consider the curve $\gamma^\prime_1$ tangent to $\gamma^\prime$ at the point $p_0$ having $2H<k(\gamma^\prime_1)<k$ (with respect to the interior normal to $\Omega$), notice that $\gamma^\prime$ is outside of $\Omega$. 

Let $\Gamma$ be the curve joining the endpoints o $\gamma^\prime_1$ having curvature $k(\gamma^\prime_1)$ with respect to the outward pointing unit normal vector to $\Omega$. Finally consider the sub-arc $\Gamma^\prime\subset\Gamma$ which lies in $\Omega$. The sub-arc $\Gamma^\prime$ intersects the arc $\gamma^\prime$ in two points, these two points determine a segment of $\gamma^\prime$, which we will call $\gamma^\prime$ again. Now consider the domain $\Delta\subset\Omega$ bounded by $\Gamma^\prime$ and $\gamma^\prime$, where $k(\gamma^\prime)>2H$ and $k(\Gamma^\prime)>2H$ wit respect to the interior unit normal of $\Delta$. 

\begin{claim}
The function $u=u\vert_{\Gamma^\prime}$ satisfies $\vert u\vert\leq C$, for some constant $C$.
\end{claim}
\textbf{Proof of the Claim.} Denote by $\Sigma_u$ the graph of $u$ over $\Delta$. As $\Sigma_u$ is transverse to the Killing field $\xi$, then $\Sigma_u$ is stable, therefore we have curvature estimates.

Now suppose that there is a sequence of points $(p_n)$ in $\Delta$ such that $u(p_n)$ goes to $-\infty$. Then the tangent planes $T_{u(p_n)}\Sigma_u$ becomes vertical when $n$ goes to $+\infty$. Using the curvature estimates, we obtain a small fixed $\delta>0$, such that $\Sigma_u$ is locally the graph of a function over a ball of radius $\delta$, centered at the origin of $T_{u(p_n)}\Sigma_u$. From Theorem \ref{S1}, $\Gamma^\prime$ must have curvature $k(\Gamma^\prime)=-2H$ which contradicts our assumption, analogously in the case $u(p_n)$ goes to $+\infty$, we get a contradiction, proving the claim. 

It was shown in \cite{Dajczer}, that the Dirichlet problem 
\begin{equation*}
\left\{
  \begin{array}{ll}
    L(u)=2H, \hspace{.3cm} \textnormal{in} \hspace{.2cm} \Delta & \hbox{} \\
    u\vert_{\Gamma^\prime}=C, \hspace{.3cm} u\vert_{\gamma^\prime}=m & \hbox{} 
  \end{array}
\right.
\end{equation*}
has a sub-solution, so there is a constant $c=c(\Omega)$ such that
$$u\geq m-c.$$

On the other hand, for the case $(ii)$, suppose there exists a decreasing sequence $\{V_n\}$ of neighbourhood of $\gamma$, that is $V_n\supset V_{n+1}$, such that for each $n$, there exists a point $p_n\in V_n$ with $u(p_n)=u_n>n$. Notice that the mean curvature vector is pointing upwards and the graph of the function $u$ has constant mean curvature $H$. Using curvature estimates, we obtain a graph $\Sigma_{u_n}$ over the tangent plane $T_{u(p_n)}\Sigma(u)$, so from Theorem \ref{S1}, $\Sigma_{u_n}$ converges to the vertical cylinder $\pi^{-1}(\gamma)$, thus $\gamma$ must have constant geodesic curvature $-2H$, which is a contradiction. This completes the proof of the lemma.

\end{proof}

\subsection{The Jenkins-Serrin problem}\label{JS problem}
We are going assume that $H>0$  with respect to the upward unit
normal of the Killing surface $\Sigma_u$. Then if $u$ tends to $+\infty$ for any approach to a boundary arc $\gamma$, necessarily the curvature $k(\gamma)=2H$ is constant, while if $u$ tends to $-\infty$ on $\gamma$, $k(\gamma)=-2H$. Thus we must deal with non-convex domains $\Omega$ with $\partial\Omega$ piecewise $C^2$ and consists of three set of open arcs $\{A_i\}$, $\{B_i\}$ and $\{C_i\}$ satisfying $k(A_i)=2H$, $k(B_i)=-2H$ and $k(C_i)\geq2H$ respectively. The Jenkins-Serrin problem consist in to find  a solution of (\ref{ee6}) in $\Omega$ taking on the boundary values $+\infty$ on $A_i$, $-\infty$ on the $B_i$ and arbitrary continuous boundary data on the $C_i$.

In order to give a precise announcement of the Jenkins-Serrin problem, we consider the following definitions.
\begin{definition}\label{c1}(Admissible domain)
We say that a bounded domain $\Omega$ is admissible if it is simply connected and its $\partial\Omega$ is piecewise  $C^2$ and consists of three set of $C^2$ open arcs $\{A_i\}$, $\{B_i\}$ and $\{C_i\}$ satisfying $k(A_i)=2H$, $k(B_i)=-2H$ and $k(C_i)\geq2H$ respectively (with respect to the interior of $\Omega$). We suppose that no two of the arcs $A_i$ and no two of the arcs $B_i$ have a common endpoint.
\end{definition}
The Jenkins-Serrin problem is defined in the next definition.
\begin{definition}\label{c2}(The Jenkins-Serrin problem)
Given an admissible domain $\Omega$, the Jenkins-Serrin problem is to find a solution of (\ref{ee6}) in $\Omega$ which assumes the value $+\infty$ on each $A_i$, $-\infty$ on each $B_i$ and assigned continuous data on each of the open arcs $C_i$. Note that the continuous data are allowed to become unbounded at the endpoints. 
\end{definition}
\begin{remark}
A solution of the Jenkins-Serrin problem is called a Jenkins-Serrin solution.
\end{remark}

\begin{definition}\label{c3}
Let $\Omega$ be an admissible domain. We say that $\mathfrak{P}$ is an admissible polygon if $\mathfrak{P}$ is a simply domain contained in $\overline{\Omega}$ with $\partial\mathfrak{P}$ piecewise smooth consisting of arcs of constant curvature $k=\pm2H$ with vertices chosen from among the endpoints of the families $\{A_i\}$, $\{B_i\}$ and $\{C_i\}$. For an admissible $\mathfrak{P}$, let $\alpha(\mathfrak{P})$ and $\beta(\mathfrak{P})$ be the total length of the arcs $\partial\mathfrak{P}$ belonging to $\{A_i\}$ and $\{B_i\}$ respectively. Finally, let $l(\mathfrak{P})$ be the perimeter of $\mathfrak{P}$ and $A(\mathfrak{P})$ be the area of $\mathfrak{P}$.
\end{definition}

\subsection{Flux formulas}
In this section, we deal with flux formulas for $H$-sections which are the crucial tool to obtain the Jenkins-Serrin solution of the Jenkins-Serrin problem.
 
We have denoted by $X_u$ the negative of the projection via $\pi$ of the horizontal part of the normal vector $N$. \newline
Let $u\in C^2(\Omega)\cap C^1(\overline{\Omega})$ be a solution of (\ref{ee6}) in a domain $\Omega$. Then, integrating (\ref{ee6}) over $\Omega$ gives 
\begin{equation}\label{a8}
2HA(\Omega)=\int_{\partial\Omega}\langle X_u,\nu\rangle ds
\end{equation}
where $A(\Omega)$ denotes the area of $\Omega$ and $\nu$ is the outer conormal to $\partial\Omega$. The right hand integral in (\ref{a8}) is called $the$ $flux$ of $u$ across $\Omega$. Let $\gamma$ be a subarc of $\partial\Omega$ (homeomorphic to $[0,1]$). Even if $u$ is not differentiable on $\gamma$ we can define the flux of $u$ across $\gamma$ as follows.
\begin{definition}
Choose $\zeta$ to be a simple smooth curve in $\Omega$ so that $\gamma\cup\zeta$ bounds a simple connected domain $\Delta_\zeta$. We then define the flux of $u$ across $\gamma$ to be 
\begin{equation}\label{a9}
F_u(\gamma)=2HA(\Delta_\zeta)-\int_\zeta\langle X_u,\nu\rangle ds.
\end{equation}
\end{definition}
Notice that the integral in (\ref{a9}) is well defined as an improper integral. To see that, this definition is independent of $\zeta$, let $\zeta^\prime$ be another choice of curve and consider the 2-chain $R$ with oriented boundary $\zeta^\prime-\zeta$. By the Divergence Theorem and equation (\ref{ee6}) we have
\begin{equation*}
2HA(\Delta_{\zeta^\prime})-2HA(\Delta_\zeta) =\int_{\zeta^\prime}\langle X_u,\nu\rangle ds-\int_\zeta\langle X_u,\nu\rangle ds.
\end{equation*}
Therefore, the definition makes sense. Thus, if $u\in C^1(\Omega\cup\gamma)$, then
$$F_u(\gamma)=\int_\gamma\langle X_u,\nu\rangle ds.$$
Thus, we obtain the following lemma.
\begin{lemma}\label{a12}
Let $u$ be a solution of (\ref{ee6}) in a domain $\Omega$ and let $\zeta$ be a piecewise $C^1$ curve in $\overline{\Omega}$. Then
$$2HA(\Omega)=\int_{\partial\Omega}\langle X_u,\nu\rangle ds \hspace{.3cm} \textnormal{and} \hspace{.3cm} \vert \int_\zeta\langle X_u,\nu\rangle ds\vert\leq \vert\zeta\vert.$$
\end{lemma}
Now we will prove some interesting lemmas.
\begin{lemma}\label{l2} 
Let $\Omega$ be a domain bounded in part by a piecewise $C^2$ arc $\zeta$ satisfying $k(\zeta)\geq 2H$. Let $u$ be a solution of (\ref{ee6}) in $\Omega$ which is continuous in $\zeta$. Then
\begin{equation}\label{a10}
\vert\int_\zeta\langle X_u,\nu\rangle ds\vert<\vert \zeta\vert.
\end{equation}
\end{lemma}
\begin{proof}
It is suffices to prove (\ref{a10}) for a small subarc $\gamma$ of $\zeta$. To this end let $p\in\zeta$ and let $\Omega_\epsilon=\Omega\cap B_\epsilon(p)$. Then by the Theorem \ref{DP1}, there is a solution $v$ of (\ref{ee6}) in $\Omega_\epsilon$ with $v=u+1$ on $\gamma$ and $v=u$ on the remainder of the boundary. Set $w=v-u$, then by Lemma \ref{l1}
$$0<\int_{\Omega_\epsilon}\langle \nabla_0 w,X_v-X_u\rangle dv=\int_\gamma\langle X_v-X_u,\nu\rangle ds.$$
Thence $F_u(\gamma)<F_v(\gamma)\leq\vert\gamma\vert.$ 
\end{proof}
\begin{lemma}\label{a16}
Let $\Omega$ be a domain bounded in part by an arc $\gamma$ and let $u$ be a solution of (\ref{ee6}) in $\Omega$. Then,
\begin{itemize}
\item[(i)] if $u$ tends to $+\infty$ on $\gamma$, we have $\int_\gamma\langle X_u,\nu\rangle ds=\vert\gamma\vert$,
\item[(ii)] if $u$ tends to $-\infty$ on $\gamma$, we have $\int_\gamma\langle X_u,\nu\rangle ds=-\vert\gamma\vert$.
\end{itemize}
\end{lemma}
\begin{proof}
Suppose $u\rightarrow+\infty$ on $\gamma$. Notice that the upwards normal vector $N$ on the surface $\Sigma_u$ is becoming horizontal when we approach of $\gamma$. Then at points sufficiently near to $\gamma$, we have
$$\langle X_u,\nu\rangle\geq1-\epsilon,$$
where $\nu$ is the outer conormal to $\gamma\subset\partial\Omega$, and $\epsilon>0$ small enough. Consequently
$$\int_\gamma\langle X_u,\nu\rangle ds\geq\int_\gamma(1-\epsilon)ds,$$
which implies 
$$\int_\gamma\langle X_u,\nu\rangle ds\geq\vert\gamma\vert.$$  

On the other hand, if $u\rightarrow-\infty$ on $\gamma$, we have
\begin{equation}\label{a11}
\langle X_u,\nu\rangle\leq-1+\epsilon,
\end{equation}
at points sufficiently near $\gamma$, and for $\epsilon>0$ small enough.

As $\langle X_u,\nu\rangle<0$, Lemma (\ref{a12}) implies
\begin{equation}\label{a13}
F_u(\gamma)\geq-\vert\gamma\vert.
\end{equation} 
Now from (\ref{a11})
$$\int_\gamma\langle X_u,\nu\rangle ds\leq\int_\gamma(-1+\epsilon)ds$$
for all $\epsilon>0$, small enough. Thence
\begin{equation}\label{a14}
F_u(\gamma)\leq-\vert\gamma\vert.
\end{equation}
We conclude from (\ref{a13}) and (\ref{a14}) that $F_u(\gamma)=-\vert\gamma\vert.$
\end{proof}
The following lemma is a simple extension of Lemma \ref{a16}.
\begin{lemma}\label{a17}
Let $\Omega$ be a domain bounded in part by an arc $\gamma$ and let $\{u_n\}$ be a sequence of solutions of (\ref{ee6}) in $\Omega$ with each $u_n$ continuous on $\gamma$. Then
\begin{itemize}
\item[(i)] if the sequence tends to $+\infty$ con compact subsets of $\gamma$ while remaining uniformly bounded on compact subsets of $\Omega$, we have
$$\lim_{n\longrightarrow+\infty}\int_\gamma\langle X_{u_n},\nu\rangle ds=\vert \gamma\vert$$
\item[(ii)] if the sequence tends to $-\infty$ con compact subsets of $\gamma$ while remaining uniformly bounded on compact subsets of $\Omega$, we have
$$\lim_{n\longrightarrow+\infty}\int_\gamma\langle X_{u_n},\nu\rangle ds=-\vert \gamma\vert$$
\end{itemize}
\end{lemma}
We also need the next lemma.
\begin{lemma}\label{b9}
Let $\Omega$ be a domain bounded in part by an arc $\gamma$ with $k(\gamma)=2H$ and let $\{u_n\}$ be a sequence of solutions of (\ref{ee6}) in $\Omega$ with each $u_n$ continuous on $\gamma$. Then if the sequence diverges to $-\infty$ uniformly on compact subsets of $\Omega$ while remaining uniformly bounded on compact subsets of $\gamma$, we have
$$\lim_{n\longrightarrow+\infty}\int_\gamma\langle X_u,\nu\rangle ds=\vert\gamma\vert.$$
\end{lemma}
\begin{proof}
Note that the sequence $\{X_{u_n}=-\pi_*(N^h_{u_n})\}$ converges uniformly to the outer normal $\nu$ on compact subsets of $\Omega$. This implies
\begin{equation}\label{e13}
\lim_{n\longrightarrow+\infty}\langle X_{u_n},\nu\rangle=1
\end{equation}
We obtain the lemma by integrating equation (\ref{e13}).
\end{proof}


\subsection{Divergence lines}
Now we focus our attention in the study of convergence  of $H$-sections $\{u_n\}$ over a domain $\Omega\subset\R^2$.

Remember that, we have denoted by $Gu_m$ the gradient in $\Nil$ of the function $u_n$. We beginning this section with the next lemma.
\begin{lemma}\label{b1}
Let $p\in\Omega$ and suppose that $\vert Gu_n(p)\vert$ is uniformly bounded. Then, there exists a subsequence of $\{v_n=u_n-u_n(p)\}$ converging uniformly to a solution of (\ref{ee6}) in a  neighbourhood of $p$ in $\Omega$. 
\end{lemma}
\begin{proof}
Notice that the surfaces $\Sigma_{v_n}$ are simply the translation of the surfaces $\Sigma_{u_n}$ along the Killing field $\xi$. So $N_{v_n}(q)=N_{u_n}(q)$ and $Gv_n(q)=Gu_n(q)$ for all $q\in\Omega$. 

As $\Sigma_{v_n}$ is stable, since the sections are transverse to the Killing field $\xi$, curvature estimates guarantee the existence of disks $D_n^\delta(p)\subset T_p\Sigma_{v_n}$ with small positive radius $\delta$, independent of $n$ (it depends only on the distance of $p$ to $\partial\Omega$), where each $\Sigma_{v_n}$ is a local graph (denoted by $\Sigma(v_n,\delta)$) over $D_n^\delta(p)$, having bounded geometry.

As $\vert Gv_n(p)\vert$ is bounded and taking into account 
$$N=-\dfrac{Gv_n}{W_n}+\xi,$$
there exists a subsequence of $N_{v_n}(p)$ which converges to a non-horizontal vector and thus, the tangent planes associated to this subsequence converge to a non-vertical plane $\Pi(p)$.

Notice that the sequence of graphs $\{\Sigma(v_n,p)\}$ have height and slope uniformly bounded, thence there is a subsequence which converges uniformly to a $H$-graph over the disk $D^\delta(p)\subset\Pi(p)$. Moreover, since $\Pi(p)$ is not vertical, there exists a geodesic ball $B(p,\overline{\delta})\subset\Omega$ centered at $p$ of radius $\overline{\delta}$ with $0<\overline{\delta}\leq\delta$, such that, the $H$-graph is a $H$-section. We conclude that there exists a neighbourhood of $p$ in $\Omega$ where a subsequence of the $\{v_n\}$ converges to a solution of equation (\ref{ee6}).
\end{proof}
The Lemma \ref{b1} motives the following definition.
\begin{definition} 
We say that the set 
$$\mathfrak{C}=\{p\in\Omega;\vert Gu_n(p)\vert \hspace{.3cm}\textnormal{is} \hspace{.3cm} \textnormal{bounded}\}$$
is the convergence set of the sequence $\{u_n\}$, and $\mathfrak{D}=\Omega-\mathfrak{C}$ is the divergence set of $\{u_n\}$.
\end{definition}
Notice that, from Lemma \ref{b1}, the convergence set $\mathfrak{C}$ is an open subset of $\Omega$.
\begin{lemma}\label{b2}
Let $\mathfrak{C}^\prime$ be a connected bounded component of $\mathfrak{C}$. Then for any $p\in\mathfrak{C}^\prime$, there exists a subsequence $\{v_{n^\prime}=u_{n^\prime}-u_{n^\prime}(p)\}$ which converges uniformly on compact subsets of $\mathfrak{C}^\prime$ to a solution of (\ref{ee6}) over $\mathfrak{C}^\prime$.
\end{lemma}
\begin{proof}
Taking a countable dense set $\{p_n\}$ in $\mathfrak{C}^\prime$. According to the Lemma \ref{b1}, each point $p_m$ has a neighbourhood where a subsequence of $\{v_n\}$ converges to an $H$-section. This convergence is uniform on compact sets of the neighbourhood. By a diagonal process, we constructed a subsequence of $\{v_n\}$ which converges uniformly on compact subsets of $\mathfrak{C}^\prime$ to an $H$-section, this happens for any $p\in\mathfrak{C}^\prime$. 
\end{proof}
Now, we study the divergent set $\mathfrak{D}$, where the sequence $\{\vert Gu_n(p)\vert\}$ diverges. If $p\in\mathfrak{D}$, $\vert Gv_n(p)\vert$ is unbounded and we can consider a subsequence $\{v_{n^\prime}\}$ of $\{v_n=u_n-u_n(p)\}$ such that $\vert Gv_{n^\prime}(p)\vert\rightarrow+\infty$. Therefore $\{N_{v_{n^\prime}}(p)\}$ converges to a horizontal vector $N(p)$.

We are going consider the identification. $\iota(\Omega)\approx\Omega$
\begin{lemma}\label{b3}
Let $p\in\Omega$ and $\{u_n\}$ be a sequence of solutions of (\ref{ee6}) in the domain $\Omega$. 
\begin{itemize}
\item[(i)] If $p\in\mathfrak{C}$, there is a subsequence of $\{v_n\}$ where $v_n=u_n-u_n(p)$, converging uniformly in a neighbourhood of $p\in\Omega$.
\item[(ii)] If $p\in\mathfrak{D}$, there is a compact arc $L_p(\widetilde{\delta})$ of curvature $2H$, containing $p$ such that after passing to a subsequence, $\{N_{v_n}(p)\}$ converges to a horizontal vector whose projection via $\pi$ is orthogonal to $L_p(\widetilde{\delta})$, having the same direction as the curvature vector to the graph of $v_n$ at $\iota(p)\approx p$.
\end{itemize}
\end{lemma}
\begin{proof}
The first part of the lemma was proved in Lemma \ref{b2}. We denote by $\Sigma_{v_n}$ the graph of the function $v_n$ over $\Omega$. Note that $N_{v_n}(q)=N_{u_{n}}(q)$ and the convergence and divergence set are the same for $\{u_n\}$ and $\{v_n\}$. 

 Since $H$-sections are stables, they have curvature estimates, see \cite{Rosenberg1}.  \newline
The curvature estimates give us an $\delta>0$ independent of $n$ ($\delta$ only depends on the distance from $p$ to $\partial\Omega$). Such that a neighbourhood of $v_n(p)$ in $\Sigma_{v_n}$ is a graph in geodesic coordinates, with height and slope uniformly bounded over the disk $D_n^{\delta}(v_n(p))$ of radius $\delta$ centered at the origin of $T_p\Sigma_{v_n}$. We call this graph $\Sigma(v_n(p),\delta)$
 
Suppose that $p\in\mathfrak{D}$. Since $\vert Gv_n(p)\vert$ is unbounded, there is a subsequence of $N_{v_n}(p)$ that converges to a horizontal vector $N(p)$. So, for this subsequence, the tangent planes $T_p\Sigma_{v_n}$ converges to a vertical plane $\Pi(p)$ and the graphs $\Sigma(v_n(p),\delta)$ converge to a constant mean curvature $H$ graph $\Sigma(p,\delta)$ over a disk of radius $\delta^\prime\leq\delta$, centered at the origin of $\Pi(p)$. By the choice of the direction of the normal vector and the choice of $H>0$, the limit of the curvature vectors of $\Sigma(v_n(p),\delta)$ has the same direction as the normal limit.

Taking the curve $L_p\subset\Omega$ passing through $p$, orthogonal to $\pi_\ast(N(p))$ with curvature $2H$ and the curvature vector at $p$ having the same direction as $\pi_\ast(N(p))$. We want to prove $\Sigma(p,\delta^\prime)\subset\pi^{-1}(L_p)$.

Since $\Sigma(p,\delta^\prime)$ is tangent to $\pi^{-1}(L_p)$ at $v_n(p)$, if $\Sigma(p,\delta^\prime)$ is in one side of $\pi^{-1}(L_p)$, by the maximum principle, we have that $\Sigma(p,\delta^\prime)\subset\pi^{-1}(L_p)$. If this is not the case, $\Sigma(p,\delta^\prime)\cap\pi^{-1}(L_p)$ is composed by $k\geq2$ curves passing through $v_n(p)$. These curves separate $\Sigma(p,\delta^\prime)$ in $2k$ components and the adjacent components lies in alternative sides of $\pi^{-1}(L_p)$. Moreover, the curvature vector alternates from pointing down to pointing up when one goes from one component to other. This implies that the normal vector to $\Sigma(p,\delta^\prime)$ points down and up. So, for n big enough, the normal vectors to $\Sigma(v_n(p),\delta)$ would point down and up, which does not occur.

Let $L_p(\widetilde{\delta})\subset\Omega$, $\delta^\prime\geq\widetilde{\delta}$ be the curve contained in $\Sigma(p,\delta^\prime)\cap\iota(L_p)$ which contains $p\approx\iota(p)$, and has length $2\widetilde{\delta}$. Since $\Sigma(p,\delta^\prime)\subset\pi^{-1}(L_p)$, we have that, for all $q\in L_p(\widetilde{\delta})$ the normal vector to $\Sigma(p,\delta^\prime)$ at $q$ is a horizontal vector normal to $L_p(\widetilde{\delta})$ at $q$.
\end{proof}
\begin{lemma}\label{b4} 
Let $\{u_n\}$ be a sequence of solutions of (\ref{ee6}) in $\Omega$. Given $p\in\mathfrak{D}$, there is a curve $L\subset\Omega$ of curvature $2H$ which passes through $p$ such that, after passing to a subsequence, the sequence of normal vectors $\{N_{u_n}\vert_L\}$ converges to a horizontal vector, whose projection via $\pi$ is normal to $L$, and having the same direction as the curvature vector of $L$. This curve $L$ contains the compact arc $L_p(\widetilde{\delta})$ given in Lemma \ref{b3}.
\end{lemma}
\begin{proof}
Let $L$ be the curve of constant curvature $2H$ in $\Omega$, which contains $L_p(\widetilde{\delta})$ joining the points of $\partial\Omega$ ($L_p(\widetilde{\delta})$ is given in the Lemma \ref{b3}). Given $p,q\in\Omega$, denote by $\overline{pq}$ the compact arc in $L$ between $p$ and $q$. We define
$$\Lambda=\{q\in L;T(q)\hspace{.1cm}\textnormal{is}\hspace{.1cm} \textnormal{true}\}$$
where
\begin{itemize}
\item $T(q)=$ there is a subsequence of $\{u_n\}$ such that $\{N_{u_n}\vert_{\overline{pq}}\}$ becomes horizontal, whose projection via $\pi$ becomes orthogonal to $L$, having the same direction as the curvature vector of $L$.
\end{itemize}
We want to prove that $\Lambda=L$. Since $p\in\Lambda$, $\Lambda$ is not empty. We will prove that $\Lambda$ is open and closed. First, we will prove that $\Lambda$ is open. Let $q$ be a point in $\Lambda$. denote by $u_{\Lambda(n)}$ the subsequence associated to $\Lambda$. Since $\Lambda\subset\mathfrak{D}$, Lemma \ref{b3} give us a curve $L_q(\widetilde{\delta})$ through $q$ such that, after passing to a subsequence $\{N_{u_{\Lambda(n)}}\vert_{L_q(\delta)}\}$ becomes horizontal and has the same direction as the curvature vector of $L_q(\delta)$. Note that this subsequence of $\{N_{u_{\Lambda(n)}}\vert_{L_q(\delta)}\}$ converges to a horizontal vector normal to $L_q(\delta)$ and to $L$ simultaneously, so $L_q(\delta)\subset L$, thus $\Lambda$ is open.

Now, we will prove that $\Lambda$ is closed. We take a convergent sequence $\{q_n\}$ in $\Lambda$, $q_n\rightarrow q\in L$. We will show that $q\in\Lambda$.\newline
For each $m$, there is a subsequence of $\{u_{\Lambda(n)}\}$ such that $\{N_{u_{\Lambda(n)}}\vert_{\overline{pq}_m}\}$ becomes horizontal with the same direction as the curvature vector of $L$ in $\overline{pq}_m$. By the diagonal process, we obtain a subsequence of $\{u_{\Lambda(n)}\}$ such that $\{N_{u_{\Lambda(n)}}\vert_{\overline{pq}_m}\}$ converges to a horizontal vector having the same direction as the curvature vector of $L$ in $\overline{pq}_m$ for all $m$. Then by Lemma \ref{b3}, we can find a curve $L_{q_m}(\delta)$  having constant curvature $2H$ through $q_m$ (for $m$ large, $\delta$ depends only on the distance from $q$ to $\partial\Omega$) such that $\{N_{u_{\Lambda(n)}}\vert_{\overline{pq}_m}\}$ converges to a horizontal vector having the same direction as the curvature vector to $L_{q_m}(\delta)$. So, $L_{q_m}\subset L$ and since $q_m\rightarrow q$, we have that, for all $m$ large enough, $q\in L_{q_m}(\delta)$. Consequently $q\in\Lambda$.

\end{proof}

An important conclusion of the Lemma \ref{b4} is that the divergence set is given by $\mathfrak{D}=\bigcup_{i\in I}L_i$, where $L_i$ is a curve called a divergence curve, having curvature $2H$. More precisely, we have the following definition.
\begin{definition}(Divergence line)
Let $L\subset\Omega$ be a curve having constant curvature $k(L)=2H$, $p\in L$ and $\{v_n=u_n-u_n(p)\}$. If there exists a subsequence of $\{N_{v_n\vert_L}\}$ which converges to a horizontal vector whose projection via $\pi$ is orthogonal to $L$, then we said that $L$ is a divergence line of $\{u_n\}$.
\end{definition}

\begin{lemma}\label{b5} 
Let $\{u_n\}$ be a sequence of solutions of (\ref{ee6}) in $\Omega$. Suppose that the divergence set $\mathfrak{D}$ of $\{u_n\}$ is composed of countable number of divergence lines. Then there exists a subsequence of $\{u_n\}$, again denoted by $\{u_n\}$ such that
\begin{itemize}
\item the divergence set of $\{u_n\}$ is composed of a countable number of pairwise disjoint divergence lines,
\end{itemize} 
\end{lemma}
\begin{proof}
Suppose that $\mathfrak{D}\neq\emptyset$ and let $L_1$ be a divergence line of $\{u_n\}$, Lemma \ref{b3} guarantees that, after passing to a subsequence, $\{N_{u_n}(q)\}$ converges to a horizontal vector whose projection via $\pi$ is orthogonal to $L_1$ at $q$ for all $q\in L$. The divergence set of this sequence is contained in the divergence set of the original sequence, so the divergence set associated to this subsequence has only countable number of lines. This subsequence is still denoted by $\{u_n\}$ and its divergence set by $\mathfrak{D}$. If there is a line $L_2\neq L_1$ in $\mathfrak{D}$, we can find  a subsequence such that $\{N_{u_n}(q)\}$ converges to a horizontal vector whose projection via $\pi$ is orthogonal to $L_2$ at $q$ for each $q\in L_2$. This implies that $L_1\cap L_2=\emptyset$. In fact, if this does not occur, we take a point $q\in L_1\cap L_2$, such that the sequence $\{N_{u_n}(q)\}$ converges to a horizontal vector whose projection is orthogonal to $L_1$ and $L_2$ at $q$ having the same direction as the curvature vector of $L_1$ and $L_2$. Then the uniqueness of a curve through $q$ having constant curvature $2H$ with a given tangent vector shows that $L_1=L_2$. We continues this process to get a subsequence of $\{u_n\}$ still denoted by $\{u_n\}$, whose divergence set is composed of a countable number of pairwise disjoint divergence lines.

\end{proof}

\begin{lemma}\label{b8} 
Let $\Omega$ be a domain bounded in part by an arc  $C$ having $k(C)\geq2H$. Let $\{u_n\}$ be an increasing or decreasing sequence of solutions  of (\ref{ee6}) in $\Omega$ with each $u_n$ continuous in $\Omega\cup C$. Suppose that $\gamma$ is an interior arc of $\Omega$ of curvature $2H$ forming part of the boundary $\mathfrak{D}$. Then $\gamma$ cannot terminate at an interior point of $C$ if $\{u_n\}$ either diverges to $\pm\infty$ or remains uniformly bounded on compact subsets of $C$. 
\end{lemma}
\begin{proof}
Suppose that $\gamma$ is an arc in $\partial\mathfrak{D}$ that terminates at an interior point $p$ of $C$. By considering only a small neighbourhood of $p$, we may assume that $C$ is $C^2$. By Theorem \ref{S1}, the sequence $\{u_n\}$ cannot diverge to $-\infty$ on $C$. Moreover, id the curvature of $C$ is not identically $2H$ and $\{u_n\}$ remains uniformly bounded on compact subsets of $C$, Lemma \ref{a15} insures that a neighbourhood of $C$ is contained in $\mathfrak{C}$ a contradiction. Hence assume $k(C)\equiv2H$. Suppose $\{u_n\}$ diverges to $+\infty$ on $C$ and there exists exactly one such $\gamma$ terminates at $p$. Let $q$ be a point of $\gamma$ close to $p$ and choose $s$ on $C$ close to $p$ so that the geodesic segment $\overline{sq}$ lies in $\mathfrak{C}$. Let $\mathfrak{T}$ be the triangle formed by $\overline{rq}$ and the constant curvature $2H$ arcs $\overline{qp}$ and $\overline{pr}$. Then by Lemma \ref{a12}
\begin{equation}\label{f1}
2HA(\mathfrak{T})=F_{u_n}(\overline{pq})+F_{u_n}(\overline{pr}) +F_{u_n}(\overline{rq})
\end{equation}
while by Lemma \ref{a17}
\begin{equation}\label{f2}
\lim_{n\longrightarrow+\infty}F_{u_n}(\overline{qp})=\vert\overline{qp}\vert, \hspace{.2cm} \lim_{n\longrightarrow+\infty}F_{u_n}(\overline{pr}) =\vert\overline{pr}\vert.
\end{equation}
From (\ref{f1}), (\ref{f2}) and Lemma \ref{l2} we conclude

\begin{equation}\label{f3}
\dfrac{2HA(\mathfrak{T})}{\vert \overline{rq}\vert}\geq\dfrac{\vert \overline{qp}\vert+\vert\overline{pr}\vert}{\vert\overline{rq}\vert} -1
\end{equation}
Keeping $p$ fixed, we move $q$ to $q^\prime$ and $r$ to $r^\prime$ along the same arcs so that $\vert\overline{q^\prime p}\vert=\lambda\vert\overline{qp}\vert$ and $\vert\overline{pr^\prime}\vert=\lambda\vert\overline{pr}\vert$ and form the triangle $\mathfrak{T}\prime$ by joining $q^\prime$ to $r^\prime$ by a geodesic. Then the left hand side of (\ref{f3}) tends to zero as $\lambda\rightarrow0$ while the right hand side of (\ref{f3}) remains uniformly positive, a contradiction. The only other possibility is that two arcs $\gamma_1$ and $\gamma_2$ terminates at $p$. Then again we can find a triangle $\mathfrak{T}\subset\mathfrak{C}$ whose edges are two constant curvature $2H$ arcs and a geodesic segment as before (perhaps $\partial\mathfrak{T}\cap C=\{p\}$). The same arguments gives a contradiction.

In the case the sequence remains uniformly bounded on compact subsets of $C$ and there is exactly on $\gamma$, we choose $r$ on $C$ so that $\mathfrak{T}$ is contained in $\mathfrak{D}$. By Lemma \ref{a15} (ii) the sequence must be divergent to $-\infty$ on $\mathfrak{D}$. We now reach a contradiction as above by using Lemma \ref{a17}. If there are two arcs terminates at $p$, then $\mathfrak{D}$ is necessarily the convex lens domain formed by $\gamma_1$ and $\gamma_2$ (that is $\gamma_2=\gamma_1^\ast$, see Remark \ref{r2}). Choose the point $q$ on $\gamma_1$ and $r$ on $\gamma_2$ and form $\mathfrak{T}$ in $\mathfrak{D}$. Then (\ref{f1}), (\ref{f2}) and (\ref{f3}) still hold and we reach a contradiction as before.
\end{proof}

We ended this section with the following theorem.
\begin{theorem}\label{MCT}(Monotone convergence theorem)
Let $\{u_n\}$ be a monotonically increasing or decreasing sequence of solutions of (\ref{ee6}) in a fixed domain $\Omega$. If the sequence is bounded in a single point of $\Omega$, there is a non-empty open set $\mathfrak{C}\subset\Omega$ such that $\{u_n\}$ converges to a solution in $\mathfrak{C}$. The convergence is uniform on compact subsets of $\mathfrak{C}$ and the divergence is uniform on compact subsets of $\mathfrak{D}=\Omega-\mathfrak{C}$. If $\mathfrak{D}$ is non-empty, $\partial\mathfrak{D}$ consists of arcs of curvature $\pm2H$ and parts of $\partial\Omega$. These arcs are convex to $\mathfrak{C}$ for increasing sequences and concave to $\mathfrak{C}$ for decreasing sequences.
\end{theorem}


\section{The main theorems}

In this section we deal with the solution for the Jenkins-Serrin problem, see Section \ref{JS problem} for notations and definitions. More precisely, we will prove the following two theorems.

\begin{theorem}\label{T1}(Main Theorem 1)
Consider the Jenkins-Serrin problem in an admissible domain $\Omega$ and suppose the family $\{C_i\}$ is empty. Then, there exists a solution for the Jenkins-Serrin problem if and only if
\begin{equation}\label{e1}
\alpha(\partial\Omega)=\beta(\partial\Omega)+2HA(\Omega)
\end{equation}
and for all admissible polygons $\mathfrak{P}$
\begin{equation}\label{e2}
2\alpha(\mathfrak{P})<l(\mathfrak{P})+2HA(\mathfrak{P}) \hspace{.3cm} \textnormal{and} \hspace{.3cm} 2\beta(\mathfrak{P})<l(\mathfrak{P})-2HA(\mathfrak{P})
\end{equation}
where $\alpha(\mathfrak{P})$ and $\beta(\mathfrak{P})$ are the total length of the arcs in $\partial\mathfrak{P}$ belonging to $\{A_i\}$ and $\{B_i\}$ respectively, finally $l(\mathfrak{P})$ denotes the perimeter of $\mathfrak{P}$ and $A(\Omega)$ denotes the area of $\mathfrak{P}$.
\end{theorem}
Following the notations of Theorem \ref{T1}, we have the next theorem.
\begin{theorem}\label{T2}(Main Theorem 2)
Consider the Jenkins-Serrin problem in an admissible domain $\Omega$ and suppose the family $\{C_i\}$ is non-empty. Then, there exists a Jenkins-Serrin solution of the Jenkins-Serrin problem if and only if
\begin{equation}\label{e3}
2\alpha(\mathfrak{P})<l(\mathfrak{P})+2HA(\mathfrak{P}) \hspace{.3cm} \textnormal{and} \hspace{.3cm} 2\beta(\mathfrak{P})<l(\mathfrak{P})-2HA(\mathfrak{P})
\end{equation}
for all admissible polygon $\mathfrak{P}$.
\end{theorem}
Before we prove the theorems, we will use the flux formulas to see that each Jenkins-Serrin solution $u$ satisfies the equations (\ref{e1}),(\ref{e2}) and (\ref{e3}) for an admissible polygon $\mathfrak{P}$. \newline
Notice that an admissible polygon $\mathfrak{P}$ can be write in the form
$$\mathfrak{P}=\{\displaystyle\cup_i A_i^\mathfrak{P}\}\cup\{\cup_j B_i^\mathfrak{P}\}\cup\{\cup_k \eta_k^\mathfrak{P}\} $$
where $A_i^\mathfrak{P}$, $B_j^\mathfrak{P}$ are arcs in $\mathfrak{P}\cap\partial\Omega$ and $\{\eta_k^\mathfrak{P}\}$ is composed by $2H$-curves in $\Omega$ and possible arcs $C_k$ in $\partial\Omega$.

For an admissible polygon $\mathfrak{P}$, the flux formulas yields
$$F_u(\partial\mathfrak{P})=2HA(\mathfrak{P}).$$
Where, we conclude that
\begin{eqnarray*}
F_u(\{\cup_i A_i^\mathfrak{P}\}) & = & -F_u(\mathfrak{P}-\{\cup_i A_i^\mathfrak{P}\}) +2HA(\mathfrak{P}) \\
F_u(\{\cup_j B_i^\mathfrak{P}\}) & = & -F_u(\mathfrak{P}-\{\cup_i B_i^\mathfrak{P}\}) +2HA(\mathfrak{P})
\end{eqnarray*}
Taking the firs equality for instance, we have
\begin{eqnarray*}
\alpha(\mathfrak{P})=\Sigma_i\vert A_i^\mathfrak{P}\vert & = & F_u(\{\cup_i A_i^\mathfrak{P}\}) \\
   & = & -F_u(\mathfrak{P}-\{\cup_i A_i^\mathfrak{P}\}) +2HA(\mathfrak{P}) \\
   & \leq & \vert F_u(\mathfrak{P}-\{\cup_i A_i^\mathfrak{P}\})\vert +2HA(\mathfrak{P}) \\
   & = & \vert F_u(\{\cup_j B_j^\mathfrak{P}\}\cup\{\eta_k^\mathfrak{P}\})\vert +2HA(\mathfrak{P}) \\
    & \leq & \vert F_u(\{\cup_j B_j^\mathfrak{P}\})\vert+ \vert F_u(\{\eta_k^\mathfrak{P}\})\vert +2HA(\mathfrak{P}) \\
    & = & \beta(\mathfrak{P})+ \vert F_u(\{\eta_k^\mathfrak{P}\})\vert +2HA(\mathfrak{P}) \\
    & = & l(\mathfrak{P})-\alpha(\mathfrak{P}) -\vert \{\eta_k^\mathfrak{P}\}\vert+ \vert F_u(\{\eta_k^\mathfrak{P}\})\vert +2HA(\mathfrak{P}) \\
    & < & l(\mathfrak{P})-\alpha(\mathfrak{P}) +2HA(\mathfrak{P})
\end{eqnarray*}
where we used that $F_u(B_j)=-\vert B_j\vert$ and $F_u(\eta_k)<\vert \eta_k\vert$. Thus 
$$2\alpha(\mathfrak{P})<l(\mathfrak{P})+2HA(\mathfrak{P}).$$
Analogously, we can see
$$2\beta(\mathfrak{P})<l(\mathfrak{P})-2HA(\mathfrak{P}).$$

On the other hand, when $\{\cup_kC_k\}=\emptyset$ and $\partial\mathfrak{P}=\partial\Omega$, from flux formulas we have
$$F_u(\{\cup_iA_i^\mathfrak{P}\}\cup\{\cup_jB_i^\mathfrak{P}\})=2HA(\Omega)$$
that is
$$F_u(\{\cup_iA_i^\mathfrak{P}\})+F_u(\{\cup_jB_i^\mathfrak{P}\})=2HA(\Omega)$$
which implies
$$\alpha(\mathfrak{P})-\beta(\mathfrak{P})=2HA(\Omega).$$

\begin{remark}\label{r2}(Lens domain)
We call a domain $D$, a lens domain, if $D$ is bounded by an arc $\gamma$ of curvature $2H$ and its geodesic reflection $\gamma^*$. 
\end{remark}

In order to prove these two last theorems, we are going prove some results. 
\begin{proposition}\label{T3}
Consider the Jenkins-Serrin problem in an admissible domain $\Omega$ and suppose the family $\{B_j\}$ is empty, $k(C_k)>2H$ and the assigned data $f$ on the arcs $C_k$ are bounded below. Then, there exists a Jenkins-Serrin solution for this Jenkins-Serrin problem if and only if
\begin{equation}\label{e4}
2\alpha(\mathfrak{P})<l(\mathfrak{P})+2HA(\mathfrak{P}) 
\end{equation}
for all admissible polygon $\mathfrak{P}$.
\end{proposition}
\begin{proof}
Let $u_n$ be the solution of (\ref{ee6}) in $\Omega$ such that
$$u_n=\left\{
  \begin{array}{ll}
    n \hspace{1.6cm} \textnormal{on} \hspace{.3cm} \cup A_i & \hbox{} \\
   \min(n,f) \hspace{.3cm} \textnormal{on} \hspace{.3cm} \cup C_i, & \hbox{}
  \end{array}
\right.
$$
such solution exits and is unique by the Dirichlet Theorem (Theorem \ref{DP1}). Moreover by the Maximum principle (Theorem \ref{MP}), the sequence $\{u_n\}$ is monotone increasing, so the Monotone Convergence Theorem (Theorem \ref{MCT}) applies. Suppose that, there is a point $p\in\Omega$ such that $u_n(p)\rightarrow+\infty$, then there is a divergence line $L(p)$ passing through $p$, thus the divergence set $\mathfrak{D}$ is not empty. Supposing that the divergence set $\mathfrak{D}$ is composed by a countable divergence lines, so from Lemma \ref{b5}, we have that $\mathfrak{D}$ is the union of pairwise disjoint divergence lines.

By Lemma \ref{b4} an interior arc of $\Omega$ which bounds $\mathfrak{D}$ must be of curvature $2H$ and by Lemma \ref{b8}, it can terminates only among the endpoints of the arcs $A_i$ or $C_i$. Moreover by Lemma \ref{a15} a neighbourhood of each $C_i$ is contained in $\mathfrak{C}$. Therefore, the boundary of each component of $\mathfrak{D}$ is an admissible polygon $\mathfrak{P}$ with vertices among those of the $A_i$ and $C_i$. Also the curvature $2H$ arcs forming the boundary which are not among the $A_i$ are concave to $\mathfrak{D}$ by Theorem \ref{S1}.

By Lemma \ref{a12} applied to each $u_n$ in $\mathfrak{P}$
\begin{equation}\label{e5}
2HA(\mathfrak{P})=F_{u_n}(\partial\mathfrak{P}-\cup^\prime A_i)+ F_{u_n}(\cup^\prime A_i)
\end{equation}
where $\cup^\prime A_i$ is the union of the arcs $A_i$ which are part of $\mathfrak{P}$. Then by Lemma \ref{a17}
\begin{equation}\label{e6}
\lim_{n\longrightarrow+\infty}F_{u_n}(\partial\mathfrak{P}-\cup^\prime A_i) =\alpha(\mathfrak{P})-l(\mathfrak{P})=-(l(\mathfrak{P})-\alpha(\mathfrak{P})).
\end{equation} 
But $\vert F_{u_n}(\cup^\prime A_i)\vert\leq \alpha(\mathfrak{P})$, hence from (\ref{e5})
$$2HA(\mathfrak{P})\leq(\alpha(\mathfrak{P})-l(\mathfrak{P}))+\alpha(\mathfrak{P})$$
that is
$$2HA(\mathfrak{P})+l(\mathfrak{P})\leq 2\alpha(\mathfrak{P}),$$
contradicting our assumption (\ref{e4}). Thus $\mathfrak{D}$ is empty and the sequence converges uniformly on compact subsets of $\Omega$ to a solution $u$. Since each $u_n$ is uniformly bounded in a neighbourhood of each $C_i$ by Lemma \ref{a15}, a standard barrier argument shows that $u=f$ on $\cup C_i$.

Since the necessity of (\ref{e4}) is clear, this completes the proof. 
\end{proof}
Similarly, we have the next theorem.
\begin{proposition}\label{T4}
Consider the Jenkins-Serrin problem in an admissible domain $\Omega$ and suppose that the family $\{A_i\}$ is empty, $k(C_i)>2H$ and the assigned data $f$ on the arcs $C_i$ are bounded above. Then, there exists a Jenkins-Serrin solution if and only if 
\begin{equation}\label{e7}
2\beta(\mathfrak{P})<l(\mathfrak{P})-2HA(\mathfrak{P})
\end{equation}
for all admissible polygon $\mathfrak{P}$. 
\end{proposition}

Now we use Proposition \ref{T3} and Proposition \ref{T4} to construct some usefull barriers in order to remove the assumption $k(C_i)>2H$.
\begin{example}\label{e}
Let $B=B_\delta(P)$ be a ball of small radius $\delta$ centered at $P$, and let $Q$ and $R$ be the "antipodal" points on $\partial B$. choose points $Q_1$ and $Q_2$ on $\partial B$ and symmetric with respect to the geodesic through $QPR$. Now let $B_1$ be an arc of curvature $-2H$ (as seem from $P$) joining $Q_1$ and $Q_2$ and set $A_1=B_i^\ast$ (the geodesic reflection of $B_1$ across its endpoints). Let $R_1$ and $R_2$ on $\partial B$ be reflections on $\partial B$ of $Q_1$ and $Q_2$ (with respect to the geodesic orthogonal to $QPR$ through $P$) and define $B_2$ and $A_2$. Then for $\delta$ small compared with $H$, the domain $B^+$ bounded by $A_1$, $A_2$ and parts of $\partial B$ satisfies the conditions of Proposition \ref{T3} and similarly, the domain $B^-$ bounded by $B_1$, $B_2$ and parts of $\partial B$ satisfies the condition of Proposition \ref{T4}.

Let $u^+$ be the solution of (\ref{ee6}) in $B^+$ with boundary values $+\infty$ on $A_1\cup A_2$ and the constant value $M$ on the remainder of the boundary. Similarly, let $u^-$ be the solution of (\ref{ee6}) in $B^-$ with boundary values $-\infty$ on $B_1\cup B_2$ and the constant value $-M$ on the remainder of the boundary.
\end{example}
With this example we obtain the following proposition.
\begin{proposition}\label{T5}
Let $\Omega$ be a domain bounded in part by an arc $\gamma$ and let $\{u_n\}$ be a sequence of solutions of (\ref{ee6}) in $\Omega$ which converges uniformly on compact subsets of $\Omega$ to a solution $u$. Suppose each $u_n$ is continuous on $\Omega\cup\gamma$, then
\begin{itemize}
\item[(i)] Suppose the boundary values of $u_n$ converges uniformly on compact subsets of $\gamma$ to a bounded limit $f$. If $k(\gamma)\geq2H$, then $u$ is continuous on $\Omega\cup\gamma$ and $u=f$ on $\gamma$.
\item[(ii)] If $k(\gamma)=2H$ and the boundary values of $u_n$ diverges uniformly to $+\infty$ on compact subsets of $\gamma$, then $u$ takes on the boundary values $+\infty$ on $\gamma$.
\item[(iii)] If $k(\gamma)=-2H$ and the boundary values of $u_n$ diverges to $-\infty$ on compact subsets of $\gamma$, then $u$ takes on the boundary values $-\infty$ on $\gamma$.
\end{itemize}
\end{proposition}
\begin{proof}
$(i)$ It suffices to prove that the sequence $\{u_n\}$ is uniformly bounded in the intersection of $\Omega$ with a neighbourhood of any interior point $P$ of $\gamma$. Orient the ball of Example \ref{e}, so that the geodesic joining $QPR$ is tangent to $\partial\Omega$. We may choose the points $Q_i$ $i=1,2$ and $\delta$ small enough so that the arc joining $Q_2$ and $R_2$ lies in a compact subset of $\Omega$. Then if $M$ is large enough 
$$u_n\leq u^+ \hspace{.1cm}\textnormal{in} \hspace{.1cm} \Omega\cap B^+ \hspace{.3cm} \textnormal{and} \hspace{.3cm} u_n\geq u^-\hspace{.1cm} \textnormal{in} \hspace{.1cm} \Omega\cap B^-.$$
Therefore, the sequence is uniformly bounded in a neighbourhood of $P$.

$(ii)$ Let $P$ be an interior point of $\gamma$. Similarly as in $(i)$, we obtain that there exists $M$ large enough so that 
$$u_n\geq -M\hspace{.1cm} \textnormal{in}\hspace{.1cm} W=B_\epsilon(P)\cap\Omega.$$
Let $v_m$ be the solution of (\ref{ee6}) in $W$ with boundary values $m$ on $\gamma\cap B_\epsilon(P)$ and $-M$ on the remaining boundary. By the Maximum Principle $u_n\geq v_m$ for $n$ sufficiently large so $u\geq v_m$ in $W$. In particular, $u(P)>m$ for every $m$ and $u$ must take on the value $+\infty$ at $P$. 

$(iii)$ Again for $P$ interior to $\gamma$, $u_n\leq M$ in $W=B_\epsilon(P)\cap\Omega$. Let $v_m$ be the solution of (\ref{ee6}) in $W$ with boundary values $-m$ on $\gamma\cap B_\epsilon(P)$ and $M$ on the remaining boundary. By the Maximum Principle $u_n\leq v_m$ for $n$ sufficiently large, so $u\leq v_m$ in $W$. Since the $v_m$ are monotonically decreasing (and converges to a solution with infinite boundary values on $\gamma\cap B_\epsilon(P)$), $u$ must takes on the value $-\infty$ at $P$.
\end{proof}

We can now extend Proposition \ref{T3} and Proposition \ref{T4} to allow the arcs $C_i$ to satisfy $k(C_i)\geq2H$. The only changed needed in the proof of Proposition \ref{T3} is to use part $(i)$ of Proposition \ref{T5} to show that the solution takes on the required boundary data on the arcs $C_i$. The extension of Proposition \ref{T4} is more delicate. Since if $k(C_i)=2H$ for some $i$, we do not know that the sequence is bounded bellow in a neighbourhood of $C_i$. However by Lemma \ref{b8} a neighbourhood of $C_i$ is either contained in $\mathfrak{C}$ or in $\mathfrak{D}$. We have already handled the former case. In the latter case, consider a component of $\mathfrak{D}$ whose boundary is an admissible polygon $\mathfrak{P}$ (with vertices among those of the $B_i$ and $C_i$) containing a subset $\cup^\prime C_i$ of the arcs $C_i$ of curvature $2H$ and a subset $\cup^\prime B_i$ of the arc $B_i$. The interior arcs of $\Omega$ which are in $\mathfrak{P}$ are convex to $\mathfrak{D}$. By Lemma \ref{a12} applied to each $u_n$ in $\mathfrak{P}$
\begin{equation}\label{e8}
2HA(\mathfrak{P})=F_{u_n}(\cup^\prime B_i)+F_{u_n}(\cup^\prime C_i) +F_{u_n}(\partial\mathfrak{P}-\cup^\prime B_i-\cup^\prime C_i)
\end{equation}
Then by Lemma \ref{a17}
\begin{equation}\label{e9}
\lim_{n\longrightarrow+\infty}F_{u_n}(\partial\mathfrak{P}-\cup^\prime B_i-\cup^\prime C_i) = l(\mathfrak{P})-\beta(\mathfrak{P})-\Sigma^\prime\vert C_i\vert
\end{equation}
and by Lemma \ref{b9}
\begin{equation}\label{e10}
\lim_{n\longrightarrow+\infty}F_{u_n}(\cup^\prime C_i)=\Sigma^\prime\vert C_i\vert
\end{equation}
But $\vert F_{u_n}(\cup^\prime B_i)\vert\leq \beta(\mathfrak{P})$, hence
$$2HA(\mathfrak{P})\geq-\beta(\mathfrak{P})+\Sigma^\prime\vert C_i\vert + (l(\mathfrak{P})-\Sigma^\prime\vert C_i\vert-\beta(\mathfrak{P})) =l(\mathfrak{P})-2\beta(\mathfrak{P})$$
contradicting our assumption (\ref{e7}). 

Thus $\mathfrak{D}$ is empty and the sequence converges uniformly on compact subsets of $\Omega$ to a solution $u$. Finally we use parts $(i)$ and $(ii)$ of Proposition \ref{T5} to show that our solution achieves the boundary values.

We state these result as
\begin{theorem}\label{T6}
Consider the Jenkins-Serrin problem in an admissible domain $\Omega$ and suppose the family $\{B_i\}$ is empty and the assigned data $f$ on the arcs $C_i$ are bounded bellow. Then, there exists a solution if and only if 
\begin{equation}\label{e11}
2\alpha(\mathfrak{P})<l(\mathfrak{P})+2HA(\mathfrak{P})
\end{equation}
for all admissible polygons $\mathfrak{P}$.
\end{theorem}
\begin{theorem}\label{T7}
Consider the Jenkins-Serrin problem in an admissible domain $\Omega$ and suppose the family $\{A_i\}$ is empty and the assigned data $f$ on the arcs $C_i$ are bounded above. Then, there exists a solution if and only if
\begin{equation}\label{e12}
2\beta(\mathfrak{P})<l(\mathfrak{P})-2HA(\mathfrak{P})
\end{equation}
for all admissible domain.
\end{theorem}
Now we prove the Main Theorem 2 (Theorem \ref{T2}). That is, we allow both families $\{A_i\}$ and $\{B_i\}$ to occur and allow the data $f$ on the $\{C_i\}$ to be unbounded both above and bellow as we approach the endpoints.
\begin{proof}(Main Theorem 2)
By Theorem \ref{T6} the first condition of (\ref{e3}) guarantees the existence of a solution $u^+$ of (\ref{ee6}) in $\Omega^\ast$ such that
$$u^+=\left\{
  \begin{array}{ll}
    +\infty \hspace{1.6cm} \textnormal{on} \hspace{.3cm} \cup A_i & \hbox{} \\
    0 \hspace{2.05cm} \textnormal{on} \hspace{.3cm} \cup B^\ast_i & \hbox{} \\
   \max(f,0) \hspace{.6cm} \textnormal{on} \hspace{.3cm} \cup C_i. & \hbox{}
  \end{array}
\right.
$$
Similarly, by Theorem \ref{T7}, the second condition of (\ref{e3}) guarantees the existence of a solution $u^-$ of (\ref{ee6}) in $\Omega^*$ such that
$$u^-=\left\{
  \begin{array}{ll}
    -\infty \hspace{1.6cm} \textnormal{on} \hspace{.3cm} \cup B_i & \hbox{} \\
    0 \hspace{2.05cm} \textnormal{on} \hspace{.3cm} \cup A_i & \hbox{} \\
   \min(f,0) \hspace{.6cm} \textnormal{on} \hspace{.3cm} \cup C_i. & \hbox{}
  \end{array}
\right.
$$
Now let $u_n$ be the solution of (\ref{ee6}) in $\Omega^*$ such that
$$u_n=\left\{
  \begin{array}{ll}
    n \hspace{1.6cm} \textnormal{on} \hspace{.3cm} \cup A_i & \hbox{} \\
    -n \hspace{1.3cm} \textnormal{on} \hspace{.3cm} \cup B^\ast_i & \hbox{} \\
   f_n \hspace{1.5cm} \textnormal{on} \hspace{.3cm} \cup C_i, & \hbox{}
  \end{array}
\right.
$$
where $f_n$ is the truncation of $f$ above by $n$ and bellow by $-n$.

By the Maximum Principle
$$u_n\leq u^+ \hspace{.1cm} \textnormal{in} \hspace{.1cm} \Omega^* \hspace{.3cm} \textnormal{and} \hspace{.3cm} u^-\leq u_n \hspace{.1cm} \textnormal{in} \hspace{.1cm} \Omega.$$
Therefore the sequence $\{u_n\}$ is uniformly bounded on compact subsets of $\Omega$ and a subsequence converges uniformly on compact subsets to a solution $u$ in $\Omega$. By Proposition \ref{T5}, $u$ takes on the boundary assigned data. The necessity of the condition (\ref{e3}) follows essentially as in the Theorem \ref{T6} and Theorem \ref{T7}. 

\end{proof}
Finally, we focus our attention in the proof of the Main Theorem 1.

\begin{proof}(Main Theorem 1)
let $v_n$ be a solution of (\ref{ee6}) in $\Omega^\ast$ with boundary values $n$ on each $A_i$ and $0$ on each $B_i^\ast$. For $0<c<n$ we define for $n\geq1$
$$E_c=\{v_n-v_0>0\} \hspace{.2cm} \textnormal{and} \hspace{.2cm} F_c=\{v_n-v_0<c\};$$
we suppress the dependence of these sets on $n$. Let $E^i_c$ and $F^i_c$ denote respectively the components of $E_c$ and $F_c$ whose closure contains respectively $A_i$ and $B^\ast_i$. By the maximum principle $E_c=\cup E^i_c$ and $F_c=\cup F^i_c$. If $c$ is sufficiently close to $n$, the sets $\{E^i_c\}$ will be distinct and disjoint (to see this, note that we can separate any two of the $A_i$ by a curve joining two of the $B^\ast_i$ on which $v_n-v_0$ is bounded away from $n$). Now we define $\mu(n)$ to be the infimum of the constants $c$ such that the sets $\{E_c^i\}$ are distinct and disjoints. The sets $\{E_\mu^i\}$ will again be distinct although there must be at least one pair $(i,j)$, $i\neq j$ such that $\overline{E}^i_\mu\cap\overline{E}^j_\mu$ is nonempty. This implies that given any $F_\mu^i$ there is some $F^j_\mu$ distinct from it. Now let $u^+_i$, $i=1,...,k$ be the solution of (\ref{ee6}) in $\Omega^\ast$ taking on the boundary values $+\infty$ on $A_i$ and $0$ on the remaining boundary. This solution exists by Theorem \ref{T6} since the solvability condition (\ref{e11}) follows trivially from (\ref{e1}) and (\ref{e2}). Also let $u^-_i$ be the solution of (\ref{ee6}) in the domain $\widetilde{\Omega}$ bounded by $\cup A_i,B^\ast_i,\cup_{j\neq i}B_j$ taking on the boundary values $-\infty$ on $\cup_{j\neq i}B_j$ and 0 on the remaining of the boundary. In order to know that the solution exists by Theorem \ref{T7}, we need to verify (\ref{e12}), thus we need only consider the admissible polygon $\overline{\mathfrak{P}}$ in $\widetilde{\Omega}$ which contain the lens domain $L$ formed by $B_i$ and $B^\ast_i$. Let $\mathfrak{P}$ be the corresponding admissible polygon for $\Omega$ formed by deleting $L$. By (\ref{e1}) and (\ref{e2}) we have
$$2\beta=2(\vert B_i\vert +\sum^\circ_{i\neq j}\vert B_j\vert) \leq l(\mathfrak{P})-2HA(\mathfrak{P})$$
or equivalently
$$2\overline{\beta}=2\sum^\circ_{i\neq j}\vert B_i\vert \leq l(\widetilde{\mathfrak{P}}))-2HA(\widetilde{\mathfrak{P}})+(2HA(L)-2\vert B_i\vert).$$
However since a solution (for example $v_n$) exits in $L$, we have by Lemma \ref{l2} $2HA(L)<2\vert B_i\vert$ so condition (\ref{e12}) is satisfied.

We now set
$$u^+(p)=\max_i\{u_i^+(p)\} \hspace{.1cm} \textnormal{in} \hspace{.1cm} \Omega^\ast \hspace{.1cm} \textnormal{and} \hspace{.1cm} u^-(p)=\min_i\{u^-_i(p)\} \hspace{.1cm} \textnormal{in} \hspace{.1cm} \Omega.$$

We note that if we compare each $u^+_i$ to a fixed bounded solution in $\Omega^\ast$ (which exists since $\Omega$ is admissible), then by the maximum principle there is a constant $N>0$ such that $u^+_i>-N$, $i=1,...,k$. Finally we set $u_n=v_n-\mu(n)$.

We now claim that
$$u_n\leq u^++M \hspace{.2cm} \textnormal{in} \hspace{.2cm} \Omega^\ast \hspace{.2cm} \textnormal{and} \hspace{.2cm} u_n\geq u^--M \hspace{.2cm} \textnormal{in} \hspace{.2cm} \Omega.$$
where $M=N+\sup_{\Omega^\ast}\vert v_0\vert$. Suppose $u_n>v_0$ at some point $p$. Then 
$$v_n-v_0>\mu(n)$$
at $p$, so that $p$ is in some $E^i_\mu$. Applying the maximum principle in the domain $E^i_\mu$, we obtain
$$u_n\leq u^+_i+N+\sup_{E^i_\mu}\vert v_o\vert\leq u^++M \hspace{.2cm} \textnormal{at} \hspace{.2cm} p.$$

On the other hand, suppose $u_n<v_0$ at some point $p\in\Omega$. Then $v_n-v_0>\mu(n)$ at $p$, so that $p$ is in some $F^i_\mu$. By what has been shown above, there is a corresponding $j=j(i)$ so that $F^i_\mu\cap F^j_\mu=\emptyset$, we obtain
$$u_n\geq u^-_j-\sup_{F^i_\mu}\vert v_o\vert\geq u^--M.$$
 
Therefore the claim is justified and the sequence $\{u_n\}$ is uniformly bounded on compact subsets of $\Omega$. By the compactness principle, a subsequence $\{u_n\}$ converges uniformly on compact subsets of $\Omega$ to a solution $u$. We still must show that $u$ takes on the required boundary values. We observe that a subsequence $\mu(n)$ diverges to $+\infty$ otherwise we can extract a subsequence converging to a finite limit $\mu_0$. Each $u_n$ would then be bounded below in $\Omega^\ast$ uniformly in $n$, and the boundary values of $u_n$ would tend uniformly on compact subsets of $\cup B^\ast_i$ to $-\mu_0$ and diverge uniformly to $+\infty$ on $\cup A_i$. Once again we could find a subsequence converging uniformly on compact subsets to a solution in $\Omega^\ast$. By Proposition \ref{T5} we would have
$$v=\left\{
  \begin{array}{ll}
    +\infty \hspace{1.6cm} \textnormal{on} \hspace{.3cm} \cup A_i & \hbox{} \\
    -\mu_0 \hspace{1.5cm} \textnormal{on} \hspace{.3cm} \cup B^\ast_i & \hbox{}
  \end{array}
\right.
$$
We can now obtain a contradiction to (\ref{e1}) by a flux argument. By Lemma \ref{a12}
\begin{equation}\label{ee14}
2HA(\Omega)=F_v(\Sigma A_i)+F_v(\Sigma B_i)
\end{equation}
while by Lemma \ref{l2} and Lemma \ref{a16}
\begin{equation}\label{ee15}
\vert F_v(\Sigma B_i)\vert<\beta \hspace{.2cm} \textnormal{and} \hspace{.2cm} F_v(\Sigma A_i)=\alpha
\end{equation}
Combining (\ref{ee14}) and (\ref{ee15}) gives $\alpha-\beta<2HA(\Omega)$, a contradiction. In the same way, we see $n-\mu(n)$ diverges to $+\infty$. Summing up we have shown that the boundary values of $u_n$, namely, $-\mu(n)$ on $\cup B_i^\ast$ and $n-\mu(n)$ on $\cup A_i$ diverges to $-\infty$ and $+\infty$ respectively. Therefore $u_n$ diverges to $-\infty$ on $\cup B_i$. Since the necessity conditions (\ref{e1}) and (\ref{e2}) is straightforward, we conclude the theorem.
\end{proof}

We ended this paper with the following maximum principle at the infinity, the prove of this theorem is similar to this one in \cite[Theorem 4.1]{Pinheiro2}  
\begin{theorem}(Maximum principle at infinity)
Let $\Omega\subset\R^2$ be an admissible domain, suppose that the family $\{C_k\}$ is nonempty, and let $u$, $v$ be solutions of the Dirichlet problem in $\Nil$, with the same continuous values on each arc $C_k\subset\partial\Omega$. Then $u=v$ on $\Omega$.
\end{theorem}

\vspace{0.5cm}

\hspace{-0.6cm}Carlos Penafiel - penafiel@im.ufrj.br\\
Universidade Federal de Rio de Janeiro\\
Instituto de Matemática e Estatística\\
Av. Athos da Silveira Ramos 149,\\
Centro de Tecnologia, Bloco C \\
Cidade Universitária - Ilha do Fundão \\
CEP 21941-909\\
Rio de Janeiro, RJ - Brasil.

\end{document}